\documentclass[letterpaper, 10 pt, conference]{ieeeconf}  

\IEEEoverridecommandlockouts                              
\pdfoutput=1

\usepackage{microtype}
\usepackage{slashbox}
\usepackage{graphicx}
\usepackage{subcaption}
\usepackage{booktabs} 
\RequirePackage[loading]{tracefnt}
\usepackage{hyperref}
\usepackage{bbm}
\usepackage{cite}
\usepackage{dsfont}
\usepackage{float}

\usepackage{amsmath}
\usepackage{amssymb}
\usepackage{mathtools}
\usepackage{multirow}
\let\proof\relax

\usepackage{amsthm}
\usepackage{subcaption}
\usepackage{algorithm}
\usepackage{algpseudocode}

\newtheorem{proposition}{Proposition}

\newtheorem{corollary}{Corollary}
\theoremstyle{definition}

\renewcommand{\qed}{\hfill\blacksquare}

\title{\LARGE \bf
Prescriptive Optimization for Adaptive Auto-insurance Pricing with Telematics Data

}

\author{Qinyang He and Yonatan Mintz
\thanks{The authors are with the Department of Industrial and Systems Engineering, University of Wisconsin-Madison, Madison, WI, 53705, USA {qhe57,ymintz}@wisc.edu}}

\begin{document}

\maketitle
\thispagestyle{empty}
\pagestyle{empty}

\begin{abstract}
Usage-based insurance (UBI) uses telematics to align premiums with risk and encourage safe driving. However, deploying these programs is challenging due to heavy-tailed claim costs, nonstationary driver behavior, and limited incentive budgets. While existing research focuses on profiling drivers, prescriptive pricing remains underexplored. We propose an optimal control framework that integrates telematics directly into dynamic pricing. Our approach (i) learns claim frequency and severity, (ii) models multi-period behavioral evolution in response to discounts, and (iii) optimizes portfolio-wide discount allocation using a Lagrangian relaxation. This decomposes the non-convex centralized problem into independent dynamical systems. We theoretically prove this relaxation's duality gap vanishes as the portfolio scales, guaranteeing asymptotic optimality. We validate our approach computationally on a simulated industry-scale portfolio. Our results demonstrate not only the computational tractability of our approach but also that it  outperforms  static baselines, reducing both expected losses and claim probabilities to benefit insurers and policyholders alike.
\end{abstract}
\section{Introduction}
The rapid advancement in telematics data collection and sensor technology has significantly transformed the automotive insurance industry, leading to more sophisticated approaches to pricing insurance premiums based on personalized risk assessment. Traditional auto insurance models, which predominantly use demographic factors and historical claims data, often lack precision in capturing real-time, individual driving behavior and associated risks. The emergence of Usage-Based Insurance (UBI), enabled by telematics technology, addresses this gap by leveraging granular driving data to personalize premiums and enhance predictive accuracy.

From a commercial perspective, the application of UBI has experienced several stages. Initially implemented as Pay-As-You-Drive (PAYD) models, where premiums were primarily based on distance traveled, UBI evolved to Pay-How-You-Drive (PHYD) and Manage-How-You-Drive (MHYD) models, incorporating behavioral metrics such as speed, acceleration, braking patterns, and even proactive driver feedback systems \cite{arumugam_survey_2019}. This evolution underscores an increasing emphasis on not only monitoring driving habits but actively influencing driver behavior to improve road safety. Usage-Based Insurance (UBI) offers significant advantages to insurance companies. By leveraging telematics data, insurers can more precisely tailor premiums to individual driving behaviors, leading to improved underwriting precision and profitability. This approach not only attracts safer drivers but also fosters customer loyalty through personalized feedback and incentives \cite{chan2024unsupervised}.

UBI has not only been advantageous to policy providers but also to program participants. Empirical evidence suggests benefits to policy holders include improved safety through behavior modification, better premium customization reflecting actual risk, and economic incentives fostering safer driving practices. For instance, research by \cite{soleymanian_sensor_2019} indicates that drivers enrolled in UBI programs significantly reduce risky behaviors, such as hard braking, due to real-time feedback and economic incentives. The main challenge of utilizing the telematics data to inform these UBI programs effectively is building a model that maps the telematics into actionable measures for decision making. Several works have attempted to address this challenge. Ayuso et al. \cite{ayuso2019improving} model claim frequency with a Poisson GLM, incorporating exposure (annual km) and telematics summaries (e.g., percentage of distance driven at night, over speed limit, in urban areas) alongside traditional rating factors to estimate expected claims. Fang et al. \cite{fang2021mocha} develop a multi-modal, multi-task LSTM that forecasts a driver's next-n-day distance, time and speed-variance through a combination of individual-level spatiotemporal tensors with dynamically updated cluster-level (``group'') tensors.

However, from a decision making perspective, the above literature only focuses on half of the end-to-end process of determining insurance pricing from UBI data which is the prediction of claim frequency and size. To date there has not been as much exploration into how to convert these predictions into an effective pricing approach. Due to the narrow profit margin, need for interpretability and many restrictions on pricing policy, mathematical optimization is usually used for deciding the optimal prices \cite{emms2007pricing, yeo2002mathematical,li2022optimal}. Since the application of UBI to insurance pricing involves frequent prediction from observation data followed by calibrated decision making by optimization methods, this falls into the broad category of sequential decision making problems with human data. Several similar approaches have been proposed in medical applications \cite{mintz_nonstationary_2020,li_adaptive_2023,adams2023planning}, HVAC management \cite{cabrera2018designing,garaza2018impact}, and have been generally related to bilevel model predictive control \cite{mintz2018control,zhang2020bi,hespanhol2018family,hespanhol2019surrogate}. In this paper we proposed a framework in the same vein as these methods with a model designed for capturing participant behavior change in UBI programs. 
The remainder of this paper is organized as follows. Section~\ref{sec:modeling} introduces the dynamic driver behavioral model and outlines the optimization approach for learning individual price sensitivities from observed telematics data. Section~\ref{sec:premium_op} formulates the multi-stage premium allocation problem, details the Lagrangian relaxation algorithm used to solve it, and establishes its theoretical guarantees. Section~\ref{sec:experiments} presents our experimental setup and numerical results, evaluating the performance and robustness of the proposed framework on an industry-scale telematics portfolio. Finally, Section V concludes the paper.

\section{modeling}
\label{sec:modeling}
In this section we describe how we model driver behavior in the UBI program and how we learn individual parameters. Our approach views drivers as utility maximizing agents that weigh their price sensitivity against the benefits that they would receive from the program. We show that these parameters can be estimated using commercial solvers.
\subsection{Driver Behavioral Model}
UBI programs are often structured in such a  way that drivers can receive discounts for how they drive. This is because, in general, individuals are more likely to enroll in UBI programs if they are guaranteed to not have their base rates increase and receive discounts.
%
%
Moreover, this enables the insurance company to more accurately price policies by decreasing the premiums of policies that incur less risk \cite{soleymanian_sensor_2019}. 
Of particular interest however is that by making these discounts conditioned on safe driving behavior, the policy providers can effectively incentivize policy holders to change their behavior and reduce their overall risk. With this in mind, our goal is to propose an adaptive pricing framework where the insurance company actively uses financial rewards to encourage safer driving behaviors and results in less life time claim costs.

There are two state quantities that the policy provider considers when making pricing decisions, these are the driver's risk (or claim probability) and expected claim amount. Let the $i \in \{1,...,N\}$ represent the index of a particular policy holding driver and suppose each driver holds a policy for up to $T$ periods. Then, let $p_{i,t}$ be the true probability that driver $i$ has a claim at time $t \in \{1,...,T\}$ and $y_{i,t}$ be the expected claim amount if the driver has one. In reality, these amounts are not known \emph{a priori} but can be estimated by machine learning method based on historic data. Let $c_{i,t} \in [0,1)$ be the discount on the original premium given to driver $i$ at time $t$ by the policy provider. We model the driver's claim probability and expected amount dynamics as:
\begin{align}
\label{behavior_dynamics}
&p_{i,t+1} = -\beta^{p}_i c_{i,t} p_{i,t} + \theta^{p}_{i}(p_{i,t} - P_{i}) + P_{i},\\
&y_{i,t+1} = -\beta^{y}_i c_{i, t} y_{i, t} + \theta^{y}_{i}(y_{i,t} - Y_{i}) + Y_{i}.
\end{align}
These equations capture two factors that characterize the driver's response to the UBI program. The first is that drivers desire to receive additional incentives in the future, and thus the driver will have reduced claim probability and amounts when a reward is given. The reduced amounts are proportional to their previous claim probability and amounts from the previous period respectively, the discount $c_{i,t}$ and unknown personalized sensitivity parameters $\beta^{p}_i,\beta^{y}_i$ as policy holders have different attitudes towards the incentive and it is important for the decision maker to capture this difference. The second term models that when no incentive is provided, the driver's behavior tends to return to a baseline level, a pattern that can be seen in many incentive-based intervention programs in various applications \cite{li2023adaptive, skinner2015schedules}. $\theta^{p}_{i},\theta^{y}_{i}$ are the rates at which the claim probability and expected amounts return to their respective baselines $P_i,Y_i$.

\subsection{Learning Price Sensitivity}
Note that before joining the program the policy providers will not know the true values of the parameters $\beta_i^p,\beta^y_i,y_{i,t},p_{i,0},P_i,Y_i$ and so must estimate their values from data. This can be done using absolute loss minimization estimation using the following optimization problem:
\begin{subequations}
    \begin{align}
    &\min \sum_{t=1}^{T}|p_{i,t}-\tilde{p}_{i,t}| + \lambda|y_{i,t}-\tilde{y}_{i,t}| \label{eq:predict_obj}\\
\notag\text{s.t. }\\
\notag&p_{i,t+1} = -\beta^{p}_i c_{i,t} p_{i,t} + \theta^{p}_{i}(p_{i,t} - P_{i}) + P_{i}, \\ & \hspace{0.25\textwidth} t \in \{0,...,T-1\} \label{eq:p_pred}\\
\notag&y_{i,t+1} = -\beta^{y}_i c_{i, t} y_{i,t} + \theta^{y}_{i}(y_{i,t} - Y_{i}) + Y_{i},\\ 
& \hspace{0.25\textwidth}t \in \{0,...,T-1\} \label{eq:y_pred}\\
&0\leq p_{i,t} \leq 1, 0\leq P_{i} \leq 1, 0 \leq y_{i} , 0\leq Y_i
\end{align}
\end{subequations}
In \eqref{eq:predict_obj}, $\tilde{p}_{i,t}$ and $\tilde{y}_{i,t}$ are the observed claim probability and amounts. Often time these values need to be extracted from raw telematics data which can be done using ML methods, we show an example of how this can be done in practice in our experiments section.  $\lambda$ is a hyperparameter that balances the magnitude between the prediction loss of the claim amounts and the claim probabilities. \ref{eq:p_pred} and \ref{eq:y_pred} are the dynamics of claim probability and amount in response to the premium discount. Note that this optimization problem is solved individually for each participating driver. The resulting estimated parameters capture how each individual driver's observed claim probability and amount trajectory match the model dynamics. This problem is a quadratically constrained quadratic program (QCQP) which can be solved by a commercial solver.

\section{Premium Optimization Problem}
\label{sec:premium_op}
Using the driver model previously described we can formulate the policy provider's premium optimization problem. Recall that the provider's overall goal is to increase their long run profit by maximizing the amount of premiums they collect while minimizing the amount of discounts and claims they pay out. If we assume that the providers have knowledge of the individual driver parameters, or high quality estimates from telematics data as described in the previous section, the policy providers can optimize premiums by solving the following $T$ period optimal control problem:

\begin{subequations}
\label{eq:premium_opt}
    \begin{align}
    &\min\limits \sum_{t=s}^{T-1}\sum_{i=1}^{N}c_{i,t}\cdot B_{i} + p_{i,t+1}\cdot y_{i,t+1} \\
\notag \text{s.t. }\\
&\notag p_{i,t+1} = -\beta^{p}_{i} c_{i,t}p_{i,t} + \theta^{p}_{i}(p_{i,t} - P_{i}) + P_{i}, \hspace{0.2cm} 
\\& \hspace{0.20\textwidth} \forall i \in I, t \in \{s,...,T-1\} \label{eq:p}\\
&\notag y_{i,t+1} = -\beta^{y} c_{i,t}y_{i,t} + \theta^{y}_{i}(y_{i,t} - Y_{i}) + Y_{i}, \hspace{0.2cm} \\& \hspace{0.20\textwidth} \forall i \in I,t \in \{s,...,T-1\} \label{eq:y}\\
&\sum_{i=1}^{N}c_{i,t+1} \cdot B_{i} \leq B,\hspace{0.09\textwidth} t \in \{s,...,T-1\} \label{eq:budget}\\
&0 \leq c_{i,t} \leq \eta, \hspace{0.09\textwidth} \forall i \in I,t \in \{s,...,T-1\} \label{eq:discount}
    \end{align}
\end{subequations}

In this formulation, $T$ is the planning ahead horizon. $B_i$ in the objective is the full premium charged to driver $i$ without discount and thus $c_{i,t}\cdot B_i$ represents the total amount of premiums rebated after a discount is applied while $p_{i,t+1}\cdot y_{i,t+1}$ represents the total amount of expected claims to pay out. Thus the objective of maximizing long term profit can be equivalently stated as minimizing the total amount of discounts rebated and claims payed. In \eqref{eq:premium_opt}, \eqref{eq:y} only $c_{i,t}$, $p_{i,t+1}$ and $y_{i,t+1}$ are variables and all other quantities are constants learned by the previous parameter learning problem. $B$ denotes the total budget the company could spend on the UBI incentive program and $\eta$ is the highest discount that could be given. \eqref{eq:budget} and \eqref{eq:discount} are realistic constraints since the policy provider can only spend limited resource on this UBI program. This formulation takes the whole driver pool into account and optimize for the best discount allocation to minimize the total loss of the company.
\subsection{Dual Lagrangian Relaxation for Multi-period Pricing}
Solving the above premium optimization problem directly using off-the-shelf commercial solvers is intractable because of the global non-convexity of the problem. The objective function contains the product of two endogenous state variables ($p_{i,t+1} \cdot y_{i,t+1}$), which are themselves generated by the bilinear transition dynamics governing claim probabilities \ref{eq:p} and amounts \ref{eq:y}. This cascades into a massive, higher-order non-convex nonlinear program. Standard commercial solvers simply cannot handle this structure at scale, either failing to converge or getting trapped in exceptionally poor local minima. While McCormick Envelopes could be used to relax these conditions, due to the number of bilinear terms (mainly due to the number of policies in the portfolio) the resulting relaxation gap would be quite large \cite{mccormick1976computability}. Moreover, using a piecewise McCormick relaxation would result in a large scale MIP that would be difficult to solve due to the scale of the problem \cite{castro2015tightening}.  However, the problem exhibits a weakly coupled structure: the individual driver trajectories are completely independent of one another and are only entangled by the global budget constraint \ref{eq:budget} at each time step. To achieve a scalable solution, we employ a Lagrangian relaxation approach. We relax the linking budget constraints by introducing a non-negative sequence of Lagrange multipliers $\lambda = \{\lambda_s, \dots, \lambda_{T-1}\}$, where $\lambda_t \ge 0$ serves as the shadow price penalizing budget violations at time $t$. The Lagrangian dual function is defined as $D(\lambda) = \min_{0 \le c_{i,t} \le \eta} \mathcal{L}(\mathbf{c}, \lambda)$. By rearranging the terms of the Lagrangian, the evaluation of $D(\lambda)$  decomposes into $N$ independent, single-driver optimization problems. Because each subproblem is a localized, sequential decision-making process over time, we can now efficiently solve them using Dynamic Programming (DP). Let $s_{i,t} = (p_{i,t}, y_{i,t})$ denote the state of driver $i$ at time $t$. The DP equations for the $i$-th  subproblem in periods $t \in \{s,...,T-1\}$ are given by:
\begin{multline}
    V_{i,t}(s_{i,t}) = \min_{c_{i,t} \in [0, \eta]} \left\{ (1 + \lambda_t)B_i c_{i,t} + p_{i,t+1}y_{i,t+1} \right.\\ \left.+ V_{i,t+1}(s_{i,t+1}) \right\}
\end{multline}

with the terminal condition $V_{i,T}(s_{i,T}) = 0$. The subsequent states $(p_{i,t+1}, y_{i,t+1})$ transition deterministically according to \ref{eq:p} and \ref{eq:y}. In this decoupled formulation, the term $(1 + \lambda_t)B_i c_{i,t}$ explicitly reflects the dynamically adjusted cost of allocating a discount, organically suppressing discount allocations when the global budget is tight.












\begin{algorithm}
\caption{Dual Lagrangian Decomposition for Multi-Period Pricing}
\label{alg:dual_decomp}
\begin{algorithmic}[1]
\Require 
    Population $N$, Horizon $T$, Budget $B$, Max Discount $\eta$;
    \Statex Driver parameters $\{B_i, P_i, Y_i, \beta_i^p, \beta_i^y, \theta_i^p, \theta_i^y\}_{i=1}^N$;
    \Statex Initial states $\{p_{i,s}, y_{i,s}\}_{i=1}^N$;
    \Statex Step size sequence $\{\alpha_k\}$; Tolerances $\epsilon_\lambda, \epsilon_{\text{feas}}$.
\Ensure Optimal discount policy $c^*_{i,t}$ and dual variables $\lambda^*_t$.

\State \textbf{Initialize:} Dual variables $\lambda_t^{(0)} \leftarrow 0, \forall t \in \{s, \dots, T-1\}$.
\State Discretize state space $(p, y)$ into grid $\mathcal{G}$.
\State $k \leftarrow 0$.

\While{not converged}
    \For{$i = 1$ to $N$}
        \State Initialize $V_{i,T}(p, y) \leftarrow 0$ for all $(p, y) \in \mathcal{G}$.
        \For{$t = T-1$ down to $s$}
            \For{each state $(p, y) \in \mathcal{G}$}
                \State Solve optimization for optimal $c$:
                \State $Q(c) = (1 + \lambda_t^{(k)}) B_i c + p' y' + V_{i,t+1}(p', y')$
                \State \quad s.t. $p', y'$ follow dynamics given $(p, y, c)$
                \State $c_{opt} \leftarrow \arg\min_{0 \le c \le \eta} Q(c)$
                \State Store policy $\pi_{i,t}(p, y) \leftarrow c_{opt}$ and Value $V_{i,t}(p, y)$.
            \EndFor
        \EndFor
        \State Set current state $(\hat{p}, \hat{y}) \leftarrow (p_{i,s}, y_{i,s})$.
        \For{$t = s$ to $T-1$}
            \State Retrieve $c_{i,t}^* \leftarrow \pi_{i,t}(\hat{p}, \hat{y})$.
            \State Update $(\hat{p}, \hat{y})$ using dynamics (Eq. 1c, 1d).
        \EndFor
    \EndFor
    \For{$t = s$ to $T-1$}
        \State Calculate aggregate budget usage: $U_t \leftarrow \sum_{i=1}^N c_{i,t}^* B_i$.
        \State Update Lagrange multiplier (Subgradient ascent):
        \State $\lambda_t^{(k+1)} \leftarrow \max \left( 0, \lambda_t^{(k)} + \alpha_k (U_t - B) \right)$.
    \EndFor
    
    \State \textbf{Check Convergence:}
    \If{$\max_t |U_t - B| \le \epsilon_{\text{feas}}$ \textbf{and} $\|\lambda^{(k+1)} - \lambda^{(k)}\| \le \epsilon_\lambda$}
        \State \textbf{break}
    \EndIf
    \State $k \leftarrow k + 1$
\EndWhile
\State \Return Policies $\{\pi_{i,t}\}_{i,t}$ and discounts $\{c_{i,t}^*\}_{i,t}$.
\end{algorithmic}
\end{algorithm}

Algorithm \ref{alg:dual_decomp} works through a two-stage structure to solve the Lagrangian dual problem $\max_{\lambda \ge 0} D(\lambda)$. The inner loop (Step 1) evaluates the dual function $D(\lambda)$ for a fixed multiplier sequence $\lambda$ by solving the $N$ independent dynamic programs. Because the state space is continuous, we discretize the claim probabilities and amounts onto a grid and solve the Bellman equations via backward induction, utilizing bilinear interpolation to estimate off-grid future values. Once the optimal local policies $\pi_{i,t}$ are extracted, the algorithm performs a forward simulation pass. Starting from the initial observed states, each driver's trajectory is simulated forward under the new policies to calculate the actual aggregate budget usage at each time step. Finally, the outer loop (Step 2) updates the shadow prices $\lambda_t$ using projected subgradient ascent. The subgradient with respect to $\lambda_t$ is exactly the budget violation $\sum_{i=1}^N c_{i,t}^* B_i - B$. If the global budget is overspent at time $t$, the shadow price $\lambda_t$ increases, making discounts mathematically more expensive for the subproblems in the next iteration. Conversely, if the budget is underutilized, the price drops toward zero. This nested process repeats until the budget constraints are safely satisfied and the dual variables converge to their optimal values

\subsection{Theoretical Guarantees}
In this section, we analyze theoretical properties of our Dual Lagrangian Decomposition algorithm
\begin{proposition}
\label{bounded_discretization_error}
Let $s=(p,y)$ be the continuous state vector defined on a compact domain $\mathcal{S}=[0,1]\times[0,Y_{max}]$. Let the grid spacing be bounded by $\Delta=\max(\Delta p,\Delta y)$. Let $V_t(s)$ be the exact optimal value function at time $t$, and let $\hat{V}_t(s)$ be the approximate value function computed on the grid, where off-grid values are computed via a bilinear interpolation operator $\mathcal{I}[\hat{V}_t](s)$. The cumulative interpolation error scales strictly linearly with the time horizon $T$. Specifically, there exists a constant $L_v>0$ such that for all $t\in\{s,\dots,T\}$:
\begin{equation}
\max_{s \in \mathcal{S}} |V_t(s) - \hat{V}_t(s)| \le L_v(T - t)\Delta\end{equation}
\end{proposition}
\proof
We proceed by backward induction. Let $\| \cdot \|_\infty$ denote the supremum norm over the compact state space $\mathcal{S}$. Define the worst-case approximation error at stage $t$ as $E_t \triangleq \| V_t - \hat{V}_t \|_\infty$. By definition, the terminal costs are zero, yielding $V_T(s) = \hat{V}_T(s) = 0, \forall s \in \mathcal{S}$. Thus, $E_T = 0 \le L_v(T-T)\Delta$, satisfying the claim. Assume the bound holds at stage $t+1$, such that $E_{t+1} \le L_v(T-t-1)\Delta$.

Let $C(s, c)$ denote the immediate cost and $h(s,c) = [p_{t+1}, y_{t+1}]^\top$ denote the state transition dynamics. The exact and approximate DP equations are:
\begin{align}
V_t(s) &= \min_{c \in [0,\eta]} \left\{ C(s, c) + V_{t+1}(f(s, c)) \right\} \\
\hat{V}_t(s) &= \min_{c \in [0,\eta]} \left\{ C(s, c) + \mathcal{I}[\hat{V}_{t+1}](f(s, c)) \right\} \\
E_t &\le \sup_{s,c} \left| V_{t+1}(h(s,c)) - \mathcal{I}[\hat{V}_{t+1}](h(s,c)) \right| \\
&= \left\| V_{t+1} - \mathcal{I}[\hat{V}_{t+1}] \right\|_\infty
\end{align}
Applying the triangle inequality, we decouple the error into an inherited approximation error and a grid interpolation error:
\begin{equation}
E_t \le \| \mathcal{I}[\hat{V}_{t+1}] - \mathcal{I}[V_{t+1}] \|_\infty + \| \mathcal{I}[V_{t+1}] - V_{t+1} \|_\infty
\end{equation}

Because bilinear interpolation computes a convex combination of adjacent grid points, the operator $\mathcal{I}$ is non-expansive in the supremum norm \cite{gordon1995stable}. Therefore, the inherited error is bounded by the inductive hypothesis:
\begin{equation}\| \mathcal{I}[\hat{V}_{t+1}] - \mathcal{I}[V_{t+1}] \|_\infty \le \| \hat{V}_{t+1} - V_{t+1} \|_\infty = E_{t+1}
\end{equation}

To bound the interpolation error $\| \mathcal{I}[V_{t+1}] - V_{t+1} \|_\infty$, we evaluate the Lipschitz constant of $V_{t+1}$. Differentiating the state dynamics yields $\frac{\partial p_{i,t+1}}{\partial p_{i,t}} = \theta_i^p - \beta_i^p c_{i,t}$. Since the mean-reversion parameter $\theta_i^p \in (0,1]$ and the discount effect $\beta_i^p c_{i,t} \ge 0$, the partial derivatives are within $[-1,1]$. Consequently, the transition mapping is non-expansive ($L_f \le 1$). Since $C(s,c)$ is Lipschitz continuous and $h(s,c)$ is non-expansive, the exact value function $V_{t+1}(s)$ preserves a uniform Lipschitz constant $L_V < \infty$ across all $t$.The standard error bound \cite{chow1991optimal} for bilinear interpolation of an $L_V$-Lipschitz function on a grid with maximal spacing $\Delta$ is given by:$\| \mathcal{I}[V_{t+1}] - V_{t+1} \|_\infty \le L_V \Delta$. Substituting these bounds yields the recursive error relation: $E_t \le E_{t+1} + L_V \Delta$. Applying the inductive hypothesis $E_{t+1} \le L_V(T-t-1)\Delta$:
\begin{equation}
E_t \le L_V(T-t-1)\Delta + L_V \Delta = L_V(T-t)\Delta \ \qed.\end{equation} 

This proposition shows that in the Dual Lagrangian algorithm the discretization error does not compound exponentially as it propagates backward through the decision epochs. Instead, the maximum approximation error accumulates strictly linearly with the time horizon $T$. For the decision-maker, this implies that the precision of the localized value functions can be deterministically controlled by refining the grid spacing $\Delta$. Note that since we know the planning horizon $T$ beforehand, we can even make the accumulated error arbitrarily small by defining $\Delta < \frac{1}{T}$ given enough computation resources. Consequently, the numerical DP remains robust, bounded, and accurate even for extended insurance contract periods without suffering from runaway inherited errors. Having established the exactness of the local subproblem evaluations, we next address the structural gap introduced by the global Lagrangian relaxation.

\begin{proposition}
\label{prop:dual_convergence}
    Consider the deterministic multi-period pricing problem with $N$ independent drivers, coupled by $M = T-s$ global budget constraints $\sum_{i=1}^N c_{i,t} B_i \le B$ for $t \in \{s,\dots,T-1\}$. Let $J^*_N$ denote the optimal primal cost and $J^D_N$ denote the optimal Lagrangian dual cost. Assume the total budget scales linearly with the population, $B = N \cdot b$. As $N \to \infty$, the relative duality gap vanishes:
    \begin{equation}\lim_{N \to \infty} \frac{J^*_N - J^D_N}{N} = 0
    \end{equation}
\end{proposition}

\proof
Let $f_i(\mathbf{c}_i)$ denote the total cost function for driver $i$, and let $\mathbf{g}_i(\mathbf{c}_i) \in \mathbb{R}^M$ denote the budget consumption vector where $g_{i,t}(\mathbf{c}_i) = c_{i,t} B_i$. The primal problem is given by:

\begin{equation}
J^*_N = \min_{\mathbf{c}} \sum_{i=1}^N f_i(\mathbf{c}_i) \quad \text{s.t.} \quad \sum_{i=1}^N \mathbf{g}_i(\mathbf{c}_i) \le \mathbf{B}
\end{equation}

By weak duality we have $J^D_N \le J^*_N$. We define the convex envelope of the local cost function as $f_i^{\text{conv}}(\mathbf{c}_i)$, which represents the pointwise supremum of all convex functions minorizing $f_i$. By standard Lagrangian duality for additively separable programs \cite{aubin1976estimates}, the optimal dual cost is exactly equal to the optimal value of the convexified primal problem:
\begin{equation}
J^D_N = \min_{\mathbf{c}} \sum_{i=1}^N f_i^{\text{conv}}(\mathbf{c}_i) \quad \text{s.t.} \quad \sum_{i=1}^N \mathbf{g}_i(\mathbf{c}_i) \le \mathbf{B}
\end{equation}

Let the non-convexity radius \cite{aubin1976estimates}, which is the maximum deviation between the true cost function and its convex envelope be: $\rho(f_i) \triangleq \sup_{\mathbf{c}_i \in \mathcal{C}} \left( f_i(\mathbf{c}_i) - f_i^{\text{conv}}(\mathbf{c}_i) \right)$. Because $\mathcal{C} = [0, \eta]^{M}$ is compact and $f_i$ is continuous, $\rho(f_i)$ is finite. Let $\rho_{\max} \triangleq \sup_{1 \le i \le N} \rho(f_i) < \infty$.To bound the duality gap $J^*_N - J^D_N$, consider the image space $\mathcal{S}_i = \{ (f_i(\mathbf{c}_i), \mathbf{g}_i(\mathbf{c}_i)) \mid \mathbf{c}_i \in \mathcal{C} \} \subset \mathbb{R}^{M+1}$ for each subproblem (driver). The convexified problem seeks an optimal vector $\mathbf{x}^{\text{conv}}$ in the Minkowski sum of the convex hulls, $\sum_{i=1}^N \text{conv}(\mathcal{S}_i)$. By the Shapley-Folkman lemma \cite{bertsekas2014constrained}, any point in $\sum_{i=1}^N \text{conv}(\mathcal{S}_i)$ can be expressed as a sum $\sum_{i=1}^N \mathbf{x}_i$ where at most $M+1$ vectors $\mathbf{x}_i$ belong strictly to the convexified sets $\text{conv}(\mathcal{S}_i) \setminus \mathcal{S}_i$, while the remaining $N - (M+1)$ vectors belong to the original non-convex sets $\mathcal{S}_i$.

Let $\bold{I}_{gap}$ denote the index set of strictly convexified subproblems, where $|\bold{I}_{gap}| \le M+1$. Let $\mathbf{c}^{\text{conv}}$ denote the optimal policy matrix for the convexified problem. We construct a strictly feasible primal policy $\bar{\mathbf{c}}$ by retaining $\bar{\mathbf{c}}_j = \mathbf{c}^{\text{conv}}_j$ for all $j \notin \bold{I}_{gap}$, for which $f_j(\bar{\mathbf{c}}_j) = f_j^{\text{conv}}(\mathbf{c}^{\text{conv}}_j)$. For the drivers $i \in \bold{I}_{gap}$, we project their policies to strict feasibility, absorbing any aggregate budget violation. Because the maximum required budget correction is bounded by $(M+1) \eta \max_i B_i$, this projection alters the objective by at most a finite constant $\kappa < \infty$. The total cost under the feasible heuristic $\bar{J}_N$ bounds the true primal cost from above ($J^*_N \le \bar{J}_N$). The difference between this feasible cost and the dual cost is bounded by the accumulated non-convexity of the convexified drivers plus the feasibility correction $\kappa$:
\begin{multline}
J^*_N - J^D_N \le \bar{J}_N - J^D_N \le \sum_{i \in \bold{I}_{gap}} \left( f_i(\bar{\mathbf{c}}_i) - f_i^{\text{conv}}(\mathbf{c}_i^{\text{conv}}) \right) \\+ \kappa \le (M+1) \rho_{\max} + \kappa
\end{multline}

Reorganizing this inequality and dividing by $N$ yields:$\frac{J^*_N - J^D_N}{N} \le \frac{(M+1) \rho_{\max} + \kappa}{N}$. Because the numerator $(M+1) \rho_{\max} + \kappa$ is a finite constant strictly independent of $N$, taking the limit as $N \to \infty$ forces the right-hand side to zero, concluding the proof. $\qed$

Proposition \ref{prop:dual_convergence} establishes the asymptotic optimality of our decentralized pricing framework. Although the problem is fundamentally non-convex, the absolute penalty for using our computationally efficient dual approach is uniformly bounded by a constant tied only to the time horizon. Because this upper bound is strictly independent of the population size, the relative suboptimality per driver drops to zero as the portfolio scales. For real-world auto-insurance portfolios encompassing tens or hundreds of thousands of active policies, this implies that our algorithm delivers a scalable pricing strategy effectively indistinguishable from the mathematically intractable true global optimum. We can simplify the rate from Proposition \ref{prop:dual_convergence} in terms of problem parameters as follows.

\begin{corollary}
    Since $\rho_{\max}$ is bounded by $\operatorname{diam}(\mathcal{Y})$ and $\kappa$ is bounded by $\max_{i \in [1,..., N]}B_i$.
    $$\frac{J^*_N - J^D_N}{N} \le \frac{(M+1)\operatorname{diam}(\mathcal{Y}) + \max_{i \in [1,..., N]}B_i}{N}$$
\end{corollary}
The proof of this corollary is immediate from Proposition \ref{prop:dual_convergence}, and shows that the key parameters outside of the portfolio size that control the relaxation error are the decision time horizon, the maximum premium size, and the range of expected claims.
\section{Experimental Results}
\label{sec:experiments}
In this section, we evaluate the effectiveness of our framework using insurance data with telematics features.
\subsection{Dataset}
We conduct all experiments on the synthetic UBI telematics portfolio constructed by So et al. \cite{so_synthetic_2021}, which emulates a large Canadian insurer’s program while preserving privacy. The released file contains 100,000 policy records with 52 variables: 11 static demographic attributes, 39 telematics features  such as total miles, rush-hour shares; per-1,000-mile hard accelerations/brakes etc., and two response variables number of claims and aggregated claim cost. The claim frequency is imbalanced, with about 95.7\% of drivers having no claim and 4.3\% having at least one claim. The synthetic portfolio is statistically anchored to a real telematics dataset with 70,000 samples collected from 2013 to 2016, but does not replicate individual records; instead, it generates driving features using Synthetic Minority Over-sampling Technique (SMOTE) \cite{chawla2002smote} and uses several neural network to generate claim amounst and frequencies given all the predictive variables. The authors performed validation by comparing models trained on real vs. synthetic data, plus Q-Q plot diagnostics, which shows closely matched patterns for key predictors, supporting the dataset’s suitability for method development and benchmarking.

\subsection{Claim Risk and Claim Amount Prediction}
As discussed in the modeling section, we require ML models that map telematics features to both claim probability and claim amount. Accordingly, we first apply one-hot encoding to all categorical variables in the synthetic dataset and train a classification model to estimate the probability $\tilde{p}$
 that a driver will file a claim. We then fit a regression model for claim amount $\tilde{y}$, augmenting the features with the predicted (or classified) claim indicator to capture dependence between frequency and severity. We evaluated a range of algorithms—including generalized linear models, tree-based ensembles, and neural networks—and found that XGBoost \cite{chen2016xgboost} delivered the strongest performance on both tasks. This choice is further supported by XGBoost’s native facilities for modeling heavy-tailed targets such as claim costs and for addressing class imbalance in claim occurrence\cite{nielsen2016tree}. The techniques we adopted in training the XGBoost models includes stratified sampling for train/validation/test data split, oversampling the minority class and hyperparameter tuning using Optuna \cite{akiba2019optuna}. 
Table \ref{tab:cls-metrics} presents the confusion matrix and key metrics of the classification model calculated on 15,000 test samples. To use a single-number measure of binary classification quality, the Matthews Correlation Coefficient (MCC) \cite{chicco2020advantages} of our model is 0.75, qualifying it as a high-signal classifier for a rare-event task. Figure \ref{fig:feature_importance} shows the feature importance for the classification model and the features marked in red are telematics features. The plot demonstrates that telematics features are very predictive of the claim probability. The claim amount prediction model also demonstrates strong agreement between predictions and observations. As shown by the parity plot, points concentrate along the $y=x$ line across the range of claim amounts, indicating good calibration. The model explains a substantial share of variance with $R^2=0.689$ while achieving low error (RMSE=633.7, MAE=107.9).which is considered a strong result compared to similar research \cite{gangaramrao2025vehicle, jonkheijmforecasting}.
\begin{table}[ht]
\centering
\begin{tabular}{lcc}
\toprule
\textbf{Confusion Matrix} & \textbf{Predicted 0} & \textbf{Predicted 1} \\
\midrule
\textbf{True 0} & 14{,}290 (TN) & 69 (FP) \\
\textbf{True 1} & 211 (FN) & 430 (TP) \\
\bottomrule
\end{tabular}

\vspace{0.75em}

\begin{tabular}{lc}
\toprule
\textbf{Metric} & \textbf{Value} \\
\midrule
Accuracy & 0.9813 (98.13\%) \\
Precision (PPV) & 0.8617 (86.17\%) \\
Recall (TPR) & 0.6708 (67.08\%) \\
F1 score & 0.7544 (75.44\%) \\
\bottomrule
\end{tabular}
\caption{Confusion matrix and classification metrics (N=15{,}000).}
\label{tab:cls-metrics}
\end{table}

\begin{figure}
    \centering
    \includegraphics[width=0.8\columnwidth]{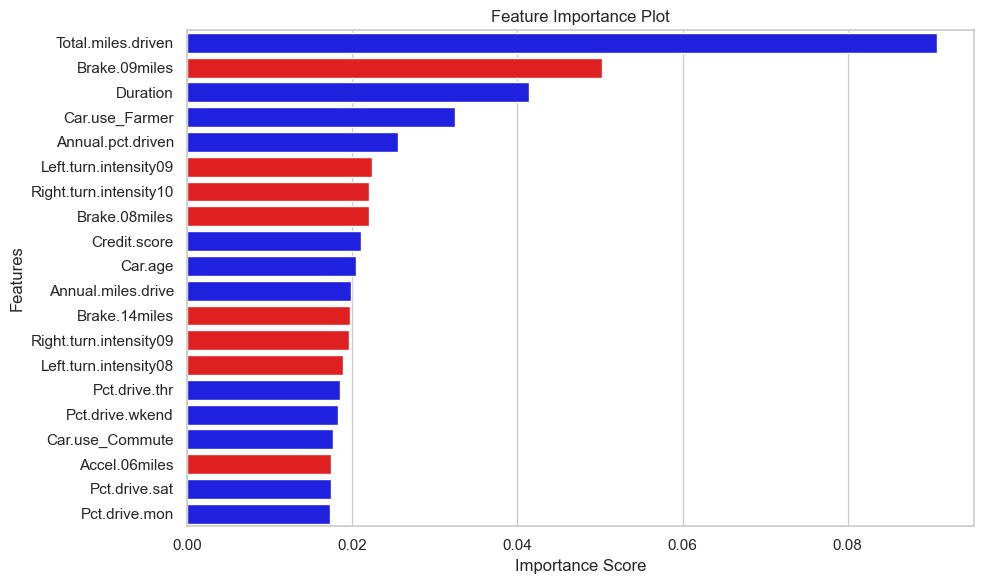}
    \caption{Feature importance for claim classification}
    \label{fig:feature_importance}
\end{figure}
\begin{figure}
    \centering
    \includegraphics[width=0.8\columnwidth]{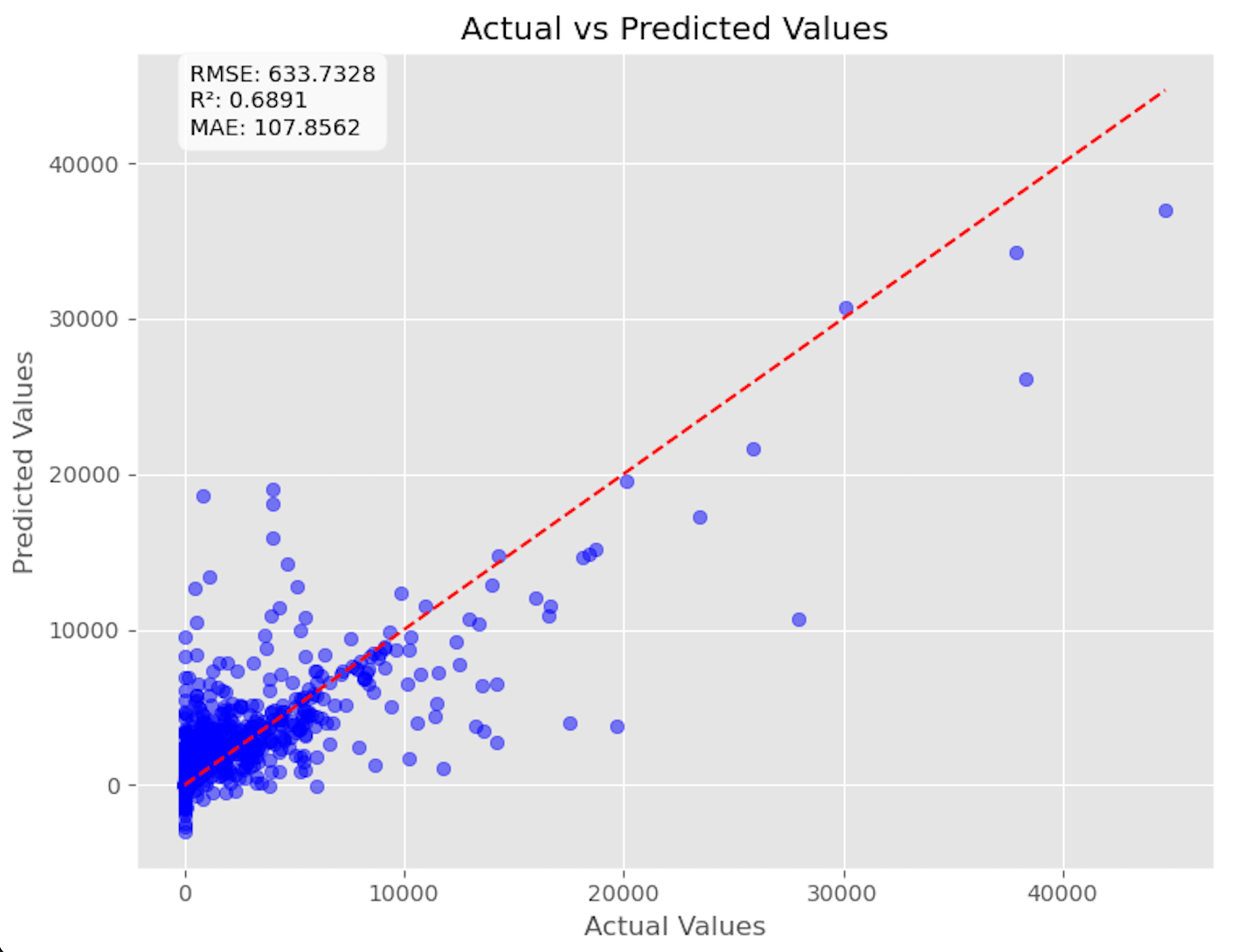}
    \caption{Parity plot for claim amount regression model}
    \label{fig:amount_parity}
\end{figure}

\subsection{Experiment Setup}
Now we describe the setup of our experiments. Notably, our pricing framework is interactive where drivers change their behavior based on their sensitivity to rewards and the insurance company provide new incentives after observing the change in drivers' behavior. Since the dataset is static covering one-time observation over a year, we need to synthesize the change in the driver's telematics features by ourselves. We assume that the driver's telematics features in the dataset (mostly representing dangerous driving behaviors such as number of sharp turns, accelerations and hard brakes) will evolve following \eqref{behavior_dynamics}, i.e, these unwanted driving behaviors will be reduced proportionally to the discount and return to baseline level if no discount is applied. To simulate individual's varying sensitivity levels, we randomly assign drivers with discount sensitivity parameter $\beta_p^{i}, \beta_y^{i}$ and rate of returning to baseline parameter $\theta^p_i, \theta^y_i$ from a uniform distribution. We denote their respective ranges as $\mathcal{B}, \Theta$. Also, the synthetic dataset does not include premium for each driver, so we assign a random initial premium $B_i$ uniformly from a range $\mathcal{P}$. We perform sensitivity analysis on how different choices of these parameters will affect the overall profitability and safety of the UBI program. The experiment runs as follows: at the beginning, we initialize the discounts by providing each driver with a uniform random discount in $[0,0.2]$, and then apply the changes to the driver's telematics features based on the discount and their sensitivity parameters. After the telematics features are updated, we use the XGBoost models to predict claim probability. $\tilde{p}$ and claim amount $\tilde{y}$. With these two estimates and discount history, the sensitivity learning optimization problem can be solved. Following that, we solve the premium optimization problem based on learned parameters using the Dual Decomposition algorithm.We set the total planning horizon $T = 8$ years as industry reports \cite{JDPower2025Shopping} shows that a typical retention in auto insurance is around this number. At iteration $s$, we solve for next $T-s$ years and discounts for calculated for year $s$ are applied. We record the amounts of reward and total expected cost for the company to evaluate the performance of our framework.
\subsection{Results and Discussion}
\begin{figure*}[t!]
  \centering
  \subfloat[$\mathcal{B}=(1,1.5), \Theta=(0.1,0.5), \mathcal{P}=(600,1000)$]{%
    \includegraphics[width=0.4\textwidth]{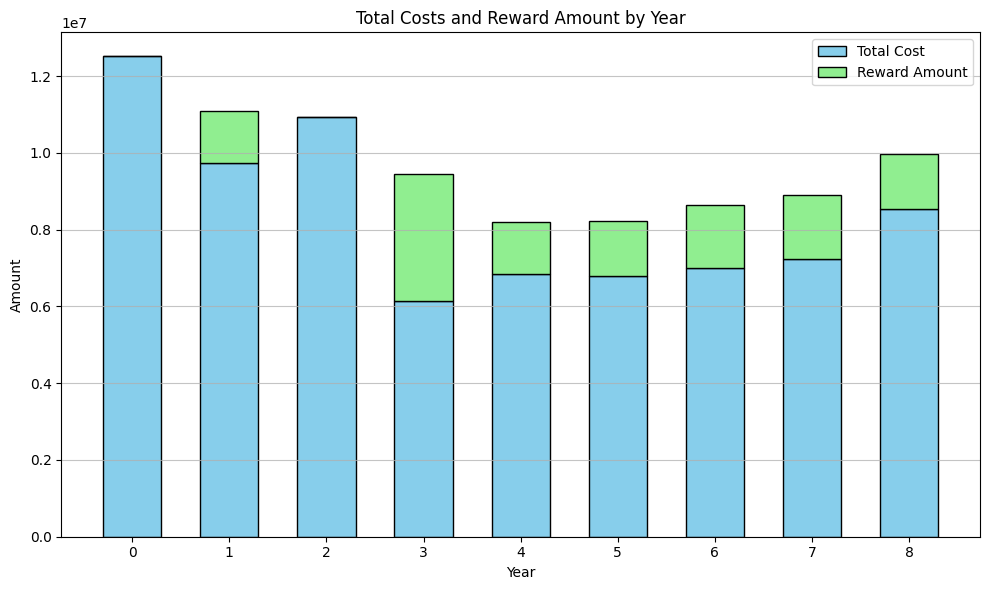}\label{fig:ieee-a}}
    \hspace{0.03\textwidth}
  \subfloat[$\mathcal{B}=(1,3), \Theta=(0.1,0.5), \mathcal{P}=(600,1000)$]{%
    \includegraphics[width=0.4\textwidth]{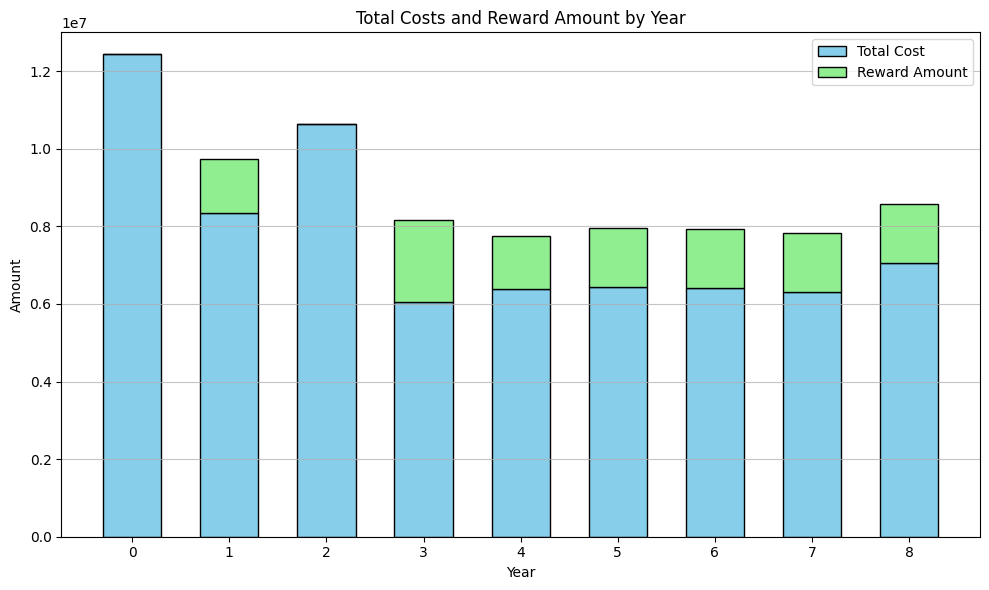}\label{fig:ieee-b}}
  \hfill
  \subfloat[$\mathcal{B}=(1,1.5), \Theta=(0.3,0.7), \mathcal{P}=(600,1000)$]{%
    \includegraphics[width=0.4\textwidth]{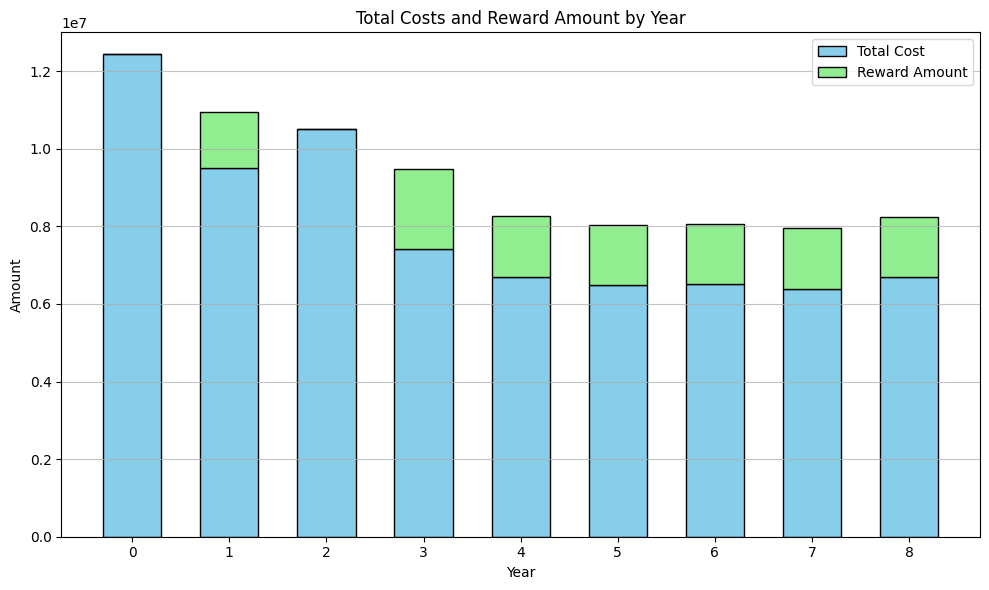}\label{fig:ieee-c}}
  \label{fig:ieee-fig1}
  \hspace{0.03\textwidth}
  \subfloat[$\mathcal{B}=(1,1.5), \Theta=(0.1,0.5), \mathcal{P}=(1000,1400)$]{%
    \includegraphics[width=0.4\textwidth]{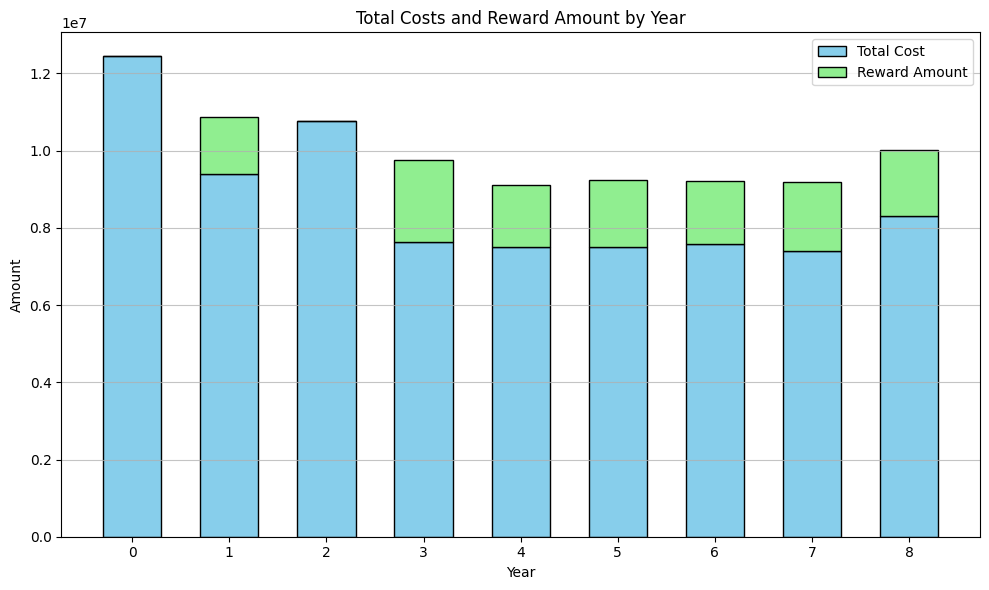}\label{fig:ieee-d}}
\caption{Total cost and reward amount by year across different parameters}
\label{fig:final_fig}
\end{figure*}
\begin{figure}
    \centering
    \includegraphics[width=0.9\columnwidth]{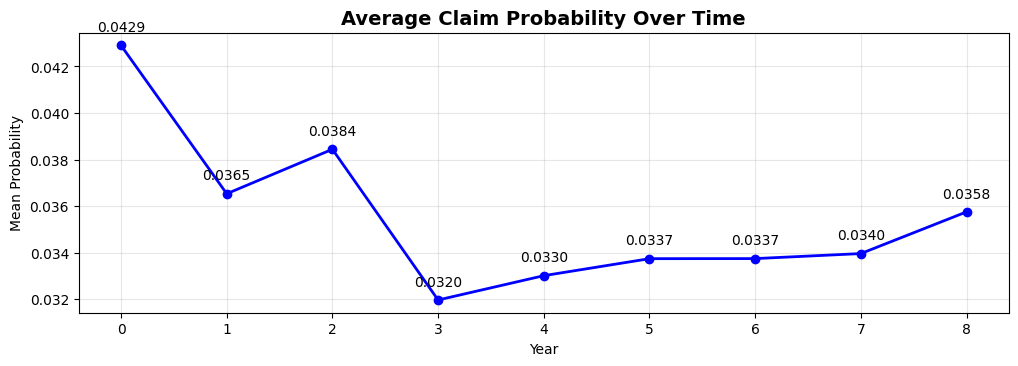}
    \caption{Changes in predicted claim probability averaged over all drivers}
    \label{fig:claim_probability}
\end{figure}
Figure \ref{fig:final_fig} reports the annual total claim cost and reward outlay under our adaptive pricing framework. In each panel, the blue bars represent expected claim costs ($\sum_{i=1}^N p_{i,t}y_{i,t}$) and the green bars represent the rewards paid (i.e., the cost of the incentive program) on top of the claim cost. The height of each stacked bar equals the company’s total loss given a fixed premium; thus, shorter bars are preferable. As we can see from the plots, our framework steadily and significantly reduces the total loss of the company compared to year 0 when no incentive program is implemented. More specifically, under the parameters in Fig. \ref{fig:ieee-a}, the total loss in year 0 is \$12.44 million, while the average across subsequent years is \$9.48 million, a 23\% reduction relative to the baseline. We also investigate how the choice of parameters would affect the algorithm performance. In Fig. \ref{fig:ieee-b}, we increase the average value for incentive sensitivity parameters $\beta_p^{i}, \beta_y^{i}$, and we see a decreased average total loss across years to \$8.57 million. This is expected because once the drivers are more sensitive to the incentive, our rewards reduce their risk of claim and claim amount more. In Fig. \ref{fig:ieee-c}, we increase the average value for the rate of returning to baseline parameters $\theta^p_i, \theta^y_i$. This also reduces the average total loss to \$8.93 million. Observing the dynamics $\theta^{p}_{i}(p_{i,t} - P_{i}) + P_{i}$, a smaller $\theta^p_i$ means that less value from previous state is preserved and the baseline value $P_i$ is more dominant. When we increase $\theta^p_i$, we increase the driver's stickiness to the encouraged behavior and thus reducing the total loss. In Fig. \ref{fig:ieee-d}, we increase the range for initial premium $B_i$, resulting in slightly increased average total loss of \$9.73 million. This is intuitive because our discount is proportional to the initial premium. When the initial premium is higher, the company pays more to reduce the same amount of claim risk which may not be ideal. Overall, sensitivity analysis indicates that our algorithm is responsive to the parameter changes and can capture nuanced individual heterogeneity. Another noteworthy point is that our framework benefits not only the insurer but also policyholders. As shown in Fig. \ref{fig:claim_probability} (under the setting of Fig. \ref{fig:ieee-a}), the average predicted claim probability across drivers declines after program enrollment—dropping from approximately 0.043 at year 0 to around 0.036 afterwards. This pattern suggests that incentives and feedback induce safer driving behavior that persists beyond the initial period, even as some attenuation appears over time. In short, the program yields mutual gains: reduced expected risk exposure for drivers and lower expected losses for the insurer.
\section{Conclusion}
We developed a finite horizon optimal control framework to address the computational and theoretical challenges of dynamic auto-insurance pricing. We used a Lagrangian relaxation approach to decouple the formulation into independent dynamic programs to solve this large problem efficiently. Our theoretical analysis proves that the duality gap of this relaxation is uniformly bounded by the time horizon, guaranteeing asymptotic optimality as the portfolio size grows. Empirically, the framework proves highly scalable for large portfolios, successfully aligning insurer profitability with improved road safety by reducing expected total losses compared to static baselines. Finally, comprehensive sensitivity analysis confirms that our dynamic pricing strategy remains highly robust across a wide range of operational budgets and underlying behavioral dynamics, cementing its value as a practical tool for insurers.
\bibliographystyle{IEEEtran}
\bibliography{IEEEabrv,refs}
\end{document}